\documentclass[12pt]{amsart}
\usepackage{amsfonts}
\usepackage{amsthm}
\usepackage{amsmath}
\usepackage{amscd}
\usepackage[latin2,utf8]{inputenc}
\usepackage[mathscr]{eucal}
\usepackage{indentfirst}
\usepackage{graphicx}
\usepackage{graphics}
\usepackage{pict2e}
\usepackage{epic}
\numberwithin{equation}{section}
\usepackage{enumitem}
\usepackage{leftidx}
\usepackage[margin=4cm]{geometry}
\usepackage{epstopdf}
\usepackage{color}
\usepackage{dsfont}
\usepackage{url,hyperref}

\theoremstyle{plain}
\newtheorem{Th}{Theorem}[section]
\newtheorem{Lemma}[Th]{Lemma}

\newtheorem{Prop}[Th]{Proposition}

\theoremstyle{definition}

\newtheorem{?}[Th]{Problem}

\providecommand{\vol}[1]{\left\lvert#1\right\rvert}
\providecommand{\abs}[1]{\left\lvert#1\right\rvert}
\providecommand{\norm}[1]{\left\lVert#1\right\rVert}

\newcommand{\R}{\mathbb{R}} 
\newcommand{\Pro}{\mathbb{P}} 
\newcommand{\Sp}{\mathbb{S}} 
\newcommand*{\ind}{\mbox{$\mathds{1}$}} 

\begin{document}

\title[A stochastic Pr\'{e}kopa-Leindler inequality]{A stochastic
  Pr\'ekopa-Leindler inequality for log-concave functions}

\author[P. Pivovarov \& J. Rebollo Bueno]{Peter Pivovarov \& Jes\'us
  Rebollo Bueno}

\address{University of Missouri \\ Department of Mathematics
  \\ Columbia MO 65201}

\email{pivovarovp@missouri.edu \& jrc65@mail.missouri.edu}


\maketitle

 \begin{abstract}
   The Brunn-Minkowski and Pr\'{e}kopa-Leindler inequalities admit a
   variety of proofs that are inspired by convexity. Nevertheless, the
   former holds for compact sets and the latter for integrable
   functions so it seems that convexity has no special signficance.
   On the other hand, it was recently shown that the Brunn-Minkowski
   inequality, specialized to convex sets, follows from a local
   stochastic dominance for naturally associated random polytopes. We
   show that for the subclass of $\log$-concave functions and
   associated stochastic approximations, a similar stochastic
   dominance underlies the Pr\'{e}kopa-Leindler inequality.
  \end{abstract}

\section{Introduction}
\label{intro}

The Brunn-Minkowski inequality governs the behavior of volume and
Minkowski addition of compacts sets $K, L\subseteq \mathbb{R}^n$:
\begin{equation}\label{Brunn-Minkowski}
  \vol{K+L}^{1/n} \geq \vol{K}^{1/n} + \vol{L}^{1/n};
\end{equation}
see \cite{schcon, garbru}.  As an isoperimetric principle,
\eqref{Brunn-Minkowski} can equivalently be stated in the form
\begin{equation}
  \label{Isoperimetric-Brunn-Minkowski}
  \vol{K+L}\geq \vol{K^* + L^*},
\end{equation}
where $A^*$ denotes the Euclidean ball of the same volume as $A$.  In
\cite{paopivrand}, it is shown that when one specializes to convex
bodies $K, L\subseteq\mathbb{R}^n$, then
\eqref{Isoperimetric-Brunn-Minkowski} admits a stronger stochastic
formulation for random polytopes.  Namely, for each convex body
$K\subseteq \mathbb{R}^n$, we sample independent random vectors
$X_1,\ldots,X_N$ uniformly in $K$ and associate a random polytope
\begin{equation*}
  [K]_N = {\rm conv}\{X_1,\ldots,X_N\},
\end{equation*}
where ${\rm conv}$ means convex hull.  Then sampling independent
random vectors in each of the bodies in
\eqref{Isoperimetric-Brunn-Minkowski} (on a common underlying
probability space $(\Omega, \mathcal{A},\mathbb{P})$) leads to the
following for each $\alpha\geq 0$,
\begin{equation}\label{Probability-Brunn-Minkowski}
  \Pro(\vol{[K]_{N}+[L]_{M}}>\alpha) \geq
  \Pro(\vol{[K^*]_{N}+[L^*]_{M}}>\alpha).
\end{equation}
By the law of large numbers, as $N,M\to\infty$, the random polytopes
converge to their ambient bodies.  Thus
\eqref{Isoperimetric-Brunn-Minkowski} follows from a ``local''
stochastic dominance for random polytopes that naturally approximate
convex bodies. Inequalities for the expected volume of the polytopes
$[K]_N$ have a long history in stochastic geometry, including
Blaschke's resolution of Sylvester's four point problem
\cite{Bla:1917}, and its generalizations to higher dimensions by
Busemann \cite{Bus:1953} and Groemer \cite{gromean}.  We explicitly
mention only Groemer's inequality on random polytopes which is
recovered from \eqref{Probability-Brunn-Minkowski} when working with
one body $K$:
\begin{equation}
  \label{eqn:Groemer}
  \mathbb{E}\vol{[K]_N}\geq \mathbb{E}\vol{[K^*]_N}.
\end{equation} 
For related work, see, e.g., \cite{camcolgroanote, GT_radius,
  paopivprob, paopivrand}, \cite[Chapter 10]{schcon} and the
references therein.

The Pr\'{e}kopa-Leindler inequality \cite{leiconv, preonlog1,
  preonlog2} asserts that for integrable functions
$f,g,h:\R^n\to[0,\infty)$ and $0<\lambda<1$, if
\begin{equation}
  \label{PL_assumption}
  h(\lambda x+(1-\lambda)y) \geq f(x)^{\lambda}g(y)^{1-\lambda}
\end{equation}for all $x,y\in\R^n$, then
\begin{equation}\label{Prekopa-Leindler}
	\displaystyle\int h \geq \left(\int
        f\right)^{\lambda}\left(\int g\right)^{1-\lambda}.
\end{equation}
By taking $f=\mathds{1}_K$, $g=\mathds{1}_L$, and
$h=\mathds{1}_{\lambda K+(1-\lambda) L}$ \eqref{Prekopa-Leindler}
implies \eqref{Brunn-Minkowski}.  In this sense the Pr\'ekopa-Leindler
inequality can be viewed as a functional extension of the
Brunn-Minkowski inequality. Related variations of
\eqref{Prekopa-Leindler} provide a basis for functional versions of
Brunn's principle; see Borell \cite{borconv}, Rinott
\cite{rinonconvexity}, Brascamp and Lieb \cite{braliebestconstants,
  bralieonextension}. The reach of the Pr\'{e}kopa-Leindler inequality
now extends into various branches of analysis, geometry, probability,
information theory, among other fields; see, e.g.,
\cite{garbru,CE_PL,CEK_12,CEM_PL,CEMS,madmelzu} and the references
therein.

As for the Brunn-Minkowski inequality, our focus here is on a form of
\eqref{Prekopa-Leindler} involving symmetric decreasing rearrangements
$f^*$ and $g^*$ (see \S \ref{prelim} for definitions). Brascamp and
Lieb proved \eqref{Prekopa-Leindler} via rearrangement inequalities
through the reverse Young inequality \cite{braliebestconstants}. More
recently, Melbourne \cite{mel} derived families of rearrangement
inequalities refining \eqref{Prekopa-Leindler}; the version we study
here involves Borel measurable $f,g:\R^n\to[0,\infty)$, $0<\lambda< 1$
  and the sup-convolution
\begin{equation*}
    (f\star_{\lambda} g)(v) =
  {\sup}\{f^{\lambda}(x)g^{1-\lambda}(y):{v=\lambda x+(1-\lambda)y}\},
\end{equation*}for which one has
\begin{equation}
  \label{Isoperimetric-Prekopa-Leindler}
  \displaystyle\int_{\R^n}{(f\star_{\lambda} g)(v)}dv \geq
  \int_{\R^n}{(f^*\star_{\lambda} g^*)(v)}dv.
\end{equation}
We will show that a similar ``local'' stochastic dominance underlies
\eqref{Isoperimetric-Prekopa-Leindler} when one focuses only on
$\log$-concave functions $f$ and $g$, i.e., when $\log f$ and $\log g$
are concave on their supports.

To define our stochastic model, for each integrable $\log$-concave
function $f:\R^n\to[0,\infty)$, we sample independent random vectors
  $(X_1,Z_1),\ldots,(X_N,Z_N)$ in $\R^n\times[0,\infty)$ according to the
    uniform Lebesgue measure on the region under the graph of $f$:
    \begin{equation}
      \label{eqn:graph}
    G_f:=\{(x,z)\in\R^n\times[0,\infty):x\in{\rm supp}f, z\leq f(x)\}.
    \end{equation}
    We denote by $[f]_N$ the least
    $\log$-concave function supported on the set $\mathop{\rm
      conv}\{X_1,\ldots,X_N\}$ with $[f]_{N}(X_i)\geq Z_i$,
    $i=1,\ldots,N$. In other words, for the convex domain
    \begin{equation}
      \label{Stochastic-Body}
            H_{f,N}:={\rm
              conv}\{(X_1,\log{Z_1}),\ldots,(X_N,\log{Z_N})\},
    \end{equation}
    we set
\begin{equation}\label{Random-Least-Log-Concave}
  [f]_N(x) := \exp\left(\sup\{z:(x,z)\in H_{f,N}\}\right).
\end{equation}
This model of approximation works equally well for $f^*$.  Indeed,
when $f$ is $\log$-concave, the same is true of $f^*$.  Thus for
independent random samples for each of the $\log$-concave functions
$f,g, f^*, g^*$, we can now state our main result.
\begin{Th}
  \label{Theo-Stochastic-PL}
  Let $f,g:\R^n\to[0,\infty)$ be $\log$-concave functions and
    $N,M>n+1$. Then, for any $\alpha>0$,
  \begin{equation*}
    \Pro\left(\int_{\R^n}([f]_N\star_{\lambda}[g]_M)(v)dv
    >\alpha\right) \geq
    \Pro\left(\int_{\R^n}([f^*]_N\star_{\lambda}[g^*]_M)(v)dv
    >\alpha\right).
  \end{equation*}
\end{Th}

When $N, M\rightarrow\infty$, the latter implies
\eqref{Isoperimetric-Prekopa-Leindler}.  Thus, in this sense, a
``local'' stochastic dominance underlies the Pr\'{e}kopa-Leindler
inequality for $\log$-concave $f$ and $g$.

There is a solid foundation for extending geometric inequalities from
convex sets to $\log$-concave functions or more general classes, e.g.,
\cite{ball-phd, klamillog, artklamilsan, frameyfun, lehdir, milrotmix,
  bobcolfraque, floartseg}.  While a number of isoperimetric
inequalities in addition to \eqref{Brunn-Minkowski} have found
stochastic versions \cite{paopivrand}, no similar progress has been
made for stochastic functional inequalities. Theorem
\ref{Theo-Stochastic-PL} is a first step towards stochastic
isoperimetric inequalities for random functions. Working with just one
function $f$, we obtain a functional analogue of Groemer's result
\eqref{eqn:Groemer} for random polytopes.

\begin{Th}
  \label{Theo-Stochastic-Groemer}
  Let $f:\R^n\to[0,\infty)$ be an integrable $\log$-concave function
    and $N>n+1$. Then, for any $\alpha>0$,
    \begin{equation*}
      \Pro\left(\int_{\R^n}{[f]_N(x)}dx >\alpha\right)\geq
      \Pro\left(\int_{\R^n}{[f^*]_N(x)}dx >\alpha\right).
    \end{equation*}
\end{Th}

The techniques from \cite{paopivrand} leading to the stochastic
Brunn-Minkowski inequality \eqref{Probability-Brunn-Minkowski} involve
random convex sets.  New ingredients are needed to treat functions and
their stochastic approximations. Various proofs of the
Pr\'{e}kopa-Leindler \eqref{Prekopa-Leindler} inequality rely on the
Brunn-Minkowski inequality. Of particular significance to our approach
is the work of Klartag \cite{klamarg}; he derives various functional
inequalities, including \eqref{Prekopa-Leindler}, as ``marginals'' of
geometric inequalities for convex bodies in higher dimensions.  The
reduction to sets in higher dimensions fits well with the stochastic
approximation that we use.  In particular, they interface well with
multiple integral rearrangement inequalities of Rogers
\cite{Rogers_single}, Brascamp-Lieb-Luttinger \cite{BLL} and Christ
\cite{christestimates}.

\section{Preliminaries}
\label{prelim}
Let $K$ be a convex set in $\mathbb{R}^n$, $\theta$ on the unit sphere
$\Sp^{n-1}$ and $P:=P_{\theta^{\perp}}$ the orthogonal projection onto
$\theta^{\perp}$. We define $u_K:PK\rightarrow \R$ by
\begin{equation*}
  \label{upperfunction}
u_K(y):=u(K,y):= \sup\{\lambda:y+\lambda\theta\in K\}
\end{equation*}
and $\ell_K:PK\rightarrow \R$ by
\begin{equation*}
  \label{lowerfunction}
\ell_K(y):=\ell(K,y):=\inf\{\lambda:y+\lambda\theta\in K\}.
\end{equation*}
Notice that, for $K$ convex, $u_K$ and $\ell_K$ are concave and
convex, respectively.

We recall that the Steiner symmetral of a non-empty compact set
$A\subseteq\R^n$ with respect to $\theta^{\perp}$,
$S_{\theta^{\perp}}A$, is the set with the property that for each line
$l$ orthogonal to $\theta^{\perp}$ and meeting $A$, the set
$l\cap S_{\theta^{\perp}}A$ is a closed segment with midpoint on
$\theta^{\perp}$ and length equal to that of the set $l\cap A$. The
mapping $S_{\theta^{\perp}}:A\rightarrow S_{\theta^{\perp}}A$ is
called the Steiner symmetrization of $A$ with respect to
$\theta^{\perp}$. In particular, if $K$ is a convex body
\begin{equation*}
	S_{\theta^{\perp}}K = \{x+\lambda \theta:x\in PK,
        -\dfrac{u_K(x)-\ell_K(x)}{2} \leq \lambda \leq
        \dfrac{u_K(x)-\ell_K(x)}{2}\}.
\end{equation*}
This shows that $S_{\theta^{\perp}}K$ is convex, since the function
$u_K-\ell_K$ is concave. Moreover, $S_{\theta^{\perp}}K$ is symmetric with
respect to $\theta^{\perp}$, it is closed, and by Fubini's theorem it
has the same volume as $K$.

Let $A\subseteq\R^n$ be a Borel set with finite Lebesgue measure. The
symmetric rearrangement, $A^*$, of $A$ is the open ball with center at
the origin whose volume is equal to the measure of $A$. Since we
choose $A^*$ to be open, $\ind_{A^*}$ is lower semicontinuous. The
symmetric decreasing rearrangement of $\ind_A$ is defined by
$\ind^*_A=\ind_{A^*}$. We say a Borel measurable function
$f:\R^n\to[0,\infty)$ vanishes at infinity if for every $t>0$, the set
  $\{x\in\R^n:f(x)>t\}$ has finite Lebesgue measure. In such a case,
  the symmetric decreasing rearrangement $f^*$ is defined by
  \begin{equation*}
    f^*(x) = \int_0^{\infty}{\ind^*_{\{f>t\}}(x)}dt =
    \int_0^{\infty}{ \ind_{\{f>t\}^*}(x)}dt.
  \end{equation*}
  Observe that $f^*$ is radially symmetric, radially decreasing, and
  equimeasurable with $f$, i.e., $\{f>t\}$ and $\{f^*>t\}$ have the
  same measure for each $t>0$. Let $\{e_1,\ldots,e_n\}$ be an
  orthonormal basis of $\R^n$ such that $e_1=\theta$. Then, for $f$
  vanishing at infinity, the Steiner symmetral $f(\cdot|\theta)$ of
  $f$ with respect to $\theta^{\perp}$ is defined as follows: set
  $f_{(x_2,\ldots,x_n),\theta}(t)=f(t,x_2,\ldots,x_n)$ and define
  $f^*(t,x_2,\ldots,x_n|\theta):=(f_{(x_2,\ldots,x_n),\theta})^*(t)$. In
  other words, we obtain $f^*(\cdot|\theta)$ by rearranging $f$ along
  every line parallel to $\theta$. We refer to the books
  \cite{lielosanalysis,simconvexity} or the introductory notes
  \cite{burashort} for further background material on rearrangement of
  functions.

\section{Approximation of $\log$-concave functions}

\label{exact-approx}

We start by recalling an approach to derving integral inequalities for
functions by using certain higher-dimensional bodies of
revolution. This method was used by Artstein, Klartag and Milman in
\cite{artklamilsan} to extend Ball's functional Blaschke-Santal\'{o}
inequality \cite{ball-phd}; see \cite{FraMey_08, BarFra_13} for
further developments.  Such bodies were also used in the first
derivation of the functional affine isoperimetric inequality
\cite{AAKSW}; see also \cite{CFGLSW}.  As we mentioned in the
introduction, Klartag \cite{klamarg} used the method to prove the
Pr\'{e}kopa-Leindler inequality and we will review several key points
for later use.

A function $f:\R^n\to[0,\infty)$ is called {\it $\log$-concave} if
  $\log f$ is concave on its support. In accordance with the usage of
  {\it $s$-concavity} in \cite{artklamilsan, klamarg}, we say $f$ is
  {\it $s$-concave} if $f^{1/s}$ is concave on its support; this is
  not the same as other common uses of the term, e.g. \cite{borconv},
  however, it fits with the approach taken here. In particular, any
  $s$-concave function, for $s>0$, is also log-concave. A useful
  approximation of a log-concave function $f$ by $s$-concave functions
  $f_s$ is given by \begin{equation}\label{s-approximation}
    f_s(x)=\left(1+\frac{\log{f(x)}}{s}\right)_+^s,
    \end{equation}where $x_+=\max\{x,0\}$.
With this choice, $f_s\leq f$ for all $s>0$, and since a log-concave
function is continuous on its support one has $f_s\rightarrow f$
locally uniformly on $\R^n$ as $s\to\infty$.

Let $(e_1,\ldots,e_{n+s})$ be an orthonormal basis of
$\R^{n+s}=\R^n\times\R^s$. For a measurable function
$f:\R^n\to[0,\infty)$, we can associate the set
  \begin{equation}
    \label{attachedbody}
    \mathcal{K}^s_f:=\{(x,y)\in\R^n\times\R^s: x\in\overline{{\rm
        supp}f}, \abs{y}\leq f^{1/s}(x)\}, 
  \end{equation}where $\abs{\cdot}$ denotes the usual Euclidean norm.
With this terminology, Brunn's principle means that a function on
$\R^n$ is $s$-concave if and only if it is a marginal of a uniform
measure on a convex body in $\R^{n+s}$. Thus $\mathcal{K}^s_f$ is
convex if and only if $f$ is $s$-concave. Moreover,
\begin{equation*}
  f(x)=\kappa_s^{-1}\displaystyle\int_{\R^s}{\ind_{\mathcal{K}^s_f}(x,y)}dy,
\end{equation*}
where $\kappa_s$ is the volume of the $s$-dimensional Euclidean ball
and hence
\begin{equation}
  \label{Volume-Integral}
  \displaystyle\int_{\R^n}{f(x)}dx = \kappa_s^{-1}\vol{\mathcal{K}^s_f}.
\end{equation}

We also recall the notion of homothety of bodies, the $s$-Minkowski
sum of two non-negative functions on $\R^n$, and their relation with
the set \eqref{attachedbody}. Let $\lambda>0$. We define the function
$\lambda\cdot_s f:\R^n\to[0,\infty)$ to be
    \begin{equation*}
        [\lambda\cdot_sf](x)=\lambda^s f\left(\frac{x}{\lambda}\right).
    \end{equation*}	
This way $\mathcal{K}^s_{\lambda\cdot_s f}=\lambda\mathcal{K}^s_f$ so
$\lambda\cdot_s f$ is a functional analog of homothety of bodies,
where if $f$ is an $s$-concave function, so is $\lambda\cdot_s
f$. Define also the $s$-Minkowski sum of two functions
$f,g:\R^n\to[0,\infty)$ as
    \begin{equation*}
      [f\oplus_s g](v)= {\sup}
      \left\{\left(f(x)^{\frac{1}{s}}+g(y)^{\frac{1}{s}}\right)^s:
      {v=x+y},{x\in{\rm supp}(f), y\in{\rm supp}(g)}\right\}
    \end{equation*}
whenever $v\in{\rm supp}(f)+{\rm supp}(g)$. If not, we set $[f\oplus_s
  g](v)=0$. This function is $s$-concave whenever $f$ and $g$ are, and
    \begin{equation*}
       \mathcal{K}^s_{f\oplus_s g}=\mathcal{K}^s_f+\mathcal{K}^s_g.
    \end{equation*}
By (\ref{attachedbody}) the latter body is convex when $f$ and $g$
are $s$-concave. Let us also denote
\begin{eqnarray*}
  (f\star_{\lambda,s}g)(v) &:=&
  ((\lambda\cdot_sf)\oplus_s((1-\lambda)\cdot_sg))(v),
\end{eqnarray*}and, as in the introduction,
\begin{eqnarray*} 
  (f\star_{\lambda}
          g)(v) &:=& {\sup}\{f^{\lambda}(x)g^{1-\lambda}(y):v=\lambda
            x+(1-\lambda)y\}.
	\end{eqnarray*}
Therefore it follows that
\begin{equation}
  \label{eqn:Ksum}
        \mathcal{K}^s_{f\star_{\lambda,s}g} =
        \mathcal{K}^s_{\lambda\cdot_sf} +
        \mathcal{K}^s_{(1-\lambda)\cdot_sg} = \mathcal{K}^s_f +_{\lambda}
        \mathcal{K}^s_g,
    \end{equation}
where we have defined 
\begin{equation}
  \label{eqn:plus_lambda}
  K+_{\lambda} L := \lambda K + (1-\lambda)L.
\end{equation}

In this way, the Prekopa-Leindler inequality is derived in
\cite{klamarg} as a ``marginal'' of the Brunn-Minkowski inequality in
$\mathbb{R}^{n+s}$ when $s\rightarrow \infty$. We cannot directly
apply the stochastic Brunn-Minkowski inequality
\eqref{Probability-Brunn-Minkowski} to the bodies $K=\mathcal{K}^s_f$
and $L=\mathcal{K}^s_g$ as this would involve different measures in
each dimension $n+s$ and an increasing number of samples in each body.
Instead, we revisit the proof of \eqref{Probability-Brunn-Minkowski}
and study operations beyond the convex hull of points that can be used
in our stochastic approach. The following lemma provides a needed
link. Henceforth, we denote the $s$-dimensional Euclidean ball
centered at a point $x\in \mathbb{R}^n$ and $\rho\geq 0$ by
\begin{equation*}
B_{\rho}^s(x)=\{(x,\widehat{z})\in
\mathbb{R}^n\times\mathbb{R}^s:\abs{\widehat{z}}\leq \rho\}.
\end{equation*}
  
\begin{Lemma}
  \label{lemma:rotation}
  Let $w_i=(x_i,z_i)\in \R^n \times [0,\infty)$ for $i=1,\ldots,N$ and
    $s\in \mathbb{N}$. Let $T=T_{\{w_i\},s}$ be the least $s$-concave
    function supported on $\mathbb{\rm conv}\{x_1,\ldots,x_N\}$ such
    that $T(x_i)\geq z_i$, $i=1,\dots,N$, i.e.,
  \begin{equation}\label{least-s-concave}
    T(x) = \sup\{z^s\in\R:(x,z)\in{\rm
      conv}\{(x_i,r_i)\}_{i=1}^N \}, 
  \end{equation}where $r_i=z_i^{1/s}$, $i=1,\ldots,N$.  Then
  \begin{equation}
    \label{Rotation-Convex}
    \mathcal{K}^s_{T} = \mathop{\rm conv}\{B^s_{r_1}(x_1),\ldots,
    B^s_{r_N}(x_N)\}.
  \end{equation}
\end{Lemma}
\begin{proof}
For $x\in \mathop{\rm conv}\{x_1,\ldots,x_N\}$ and $\widehat{z}\in
\mathbb{R}^s$ with $(x,\widehat{z})\in \mathcal{K}^s_{T}$, there
exists non-negative $c_1,\ldots,c_N$ with $\sum_{i}c_i=1$ such that
\[
\abs{\widehat{z}} \leq T(x)^{1/s} =
\displaystyle\sum\limits_{i=1}^N{c_ir_i},
\]
and $x=\sum\limits_{i=1}^N{c_ix_i}$. Thus, denoting
$r:=\sum\limits_{i=1}^N{c_ir_i}$, we have
\[
(x,\widehat{z})\in B^s_{r}(x) \subseteq \mathop{\rm
  conv}\{B^s_{r_1}(x_1),\ldots, B^s_{r_N}(x_N)\}.
\]
On the other hand, let $(x,\widehat{z})\in \mathop{\rm
  conv}\{B^s_{r_1}(x_1),\ldots, B^s_{r_N}(x_N)\}$ so that for some
$(c_i)$ with $c_i\geq 0$, $\sum_i c_i = 1$,
$(x,\widehat{z})=\sum\nolimits_{i=1}^N{c_i(x_i,\widehat{z}_i)}$
where $\widehat{z}_i\in \mathbb{R}^s$ and $|\widehat{z}_i|\leq r_i$,
for $i=1,\ldots,N$.  Therefore,
\[
\abs{\widehat{z}} \leq \displaystyle\sum\limits_{i=1}^N{c_ir_i} \leq
T(x)^{1/s},
\]
so $(x,\widehat{z})\in\mathcal{K}^s_{T}$ as desired.
\end{proof}

\section{Rearrangements and Steiner convexity}


\label{Steiner-Convexity}

When an isoperimetric principle admits a proof by symmetrization, like
\eqref{Isoperimetric-Brunn-Minkowski} for example, it is often
meaningful to instead carry out such symmetrization on a suitable
(product) probability space.  In \cite{paopivrand}, a variety of
isoperimetric inequalities for convex sets are shown to admit stronger
stochastic forms. A key tool in this approach involves multiple
integral rearrangement inequalities of Rogers \cite{Rogers_single},
and Brascamp-Lieb-Luttinger \cite{BLL}. Christ's version
\cite{christestimates} of the latter is especially well-suited for
stochastic forms of isoperimetric inequalities; as in
\cite{paopivrand}, the following formulation is convenient for our
purpose.

\begin{Th}
  \label{RBLLC}
  Let $f_1,\ldots,f_N$ be non-negative integrable functions on
  $\mathbb{R}^n$ and $F:(\mathbb{R}^n)^N\rightarrow [0,\infty)$. Then
   \begin{eqnarray*}
    \lefteqn{\int\limits_{(\R^n)^N} F(x_1,\ldots,x_N)\prod_{i=1}^N
      f_i(x_i)dx_1\ldots dx_N}\\ & \geq & \int\limits_{(\R^n)^N}
    F(x_1,\ldots,x_N)\prod_{i=1}^N f^*_i(x_i)dx_1\ldots dx_N,
  \end{eqnarray*}whenever  $F$ satisfies the following condition:
  for each $\theta\in\Sp^{n-1}$ and all
  $Y:=\{y_1,\ldots,y_N\}\subseteq\theta^{\perp}$, the function
  $F_Y:\R^N\to[0,\infty)$ defined by
    \begin{equation*}
      F_{Y,\theta}(t_1,\ldots,t_N):=F(y_1+t_1\theta,\ldots,y_N+t_N\theta)
    \end{equation*} is even and quasi-convex.
\end{Th}

The condition on $F$ allows the theorem to be proved via iterated
Steiner symmetrization; as in \cite{paopivrand}, we will call such
functions $F$ {\it Steiner convex} (which differs from the terminology
in \cite{christestimates}). This condition interfaces well with
generalizations of Steiner symmetrization like shadow systems, e.g.,
\cite{rogsheext, camcolgroanote}; see \cite{paopivrand} for further
background and references. For context, we recall only several
examples before treating the functionals involved in our main
theorems.

A fundamental example of a Steiner convex function is the absolute
value of the determinant, $F(x_1,\ldots,x_n)=\abs{\mathop{\rm
    det}([x_1,\ldots,x_n])}$ (when $N=n$); this is the key property
behind Busemann's random simplex inequality \cite{Bus:1953}.  More
generally, Groemer \cite{gromean} showed that for $N>n$, the
functional
\begin{equation}
  \label{eqn:Groemer_funct}
 F(x_1,\ldots,x_N) =\vol{\mathop{\rm conv}\{x_1,\ldots,x_N\}}
\end{equation}
is also Steiner convex. While the latter examples involve points
$x_i$, analogous results hold for convex hulls of Euclidean balls
\begin{equation*}
  B_{\rho_i}(x_i):=\{u\in \mathbb{R}^n:\abs{u-x_i}\leq
  \rho_i\}.
\end{equation*}
Indeed, in \cite{Pfiefer}, Pfiefer showed that for
$\rho_1,\ldots,\rho_N\geq 0$, the functional
\begin{equation}
  \label{eqn:Pfiefer}
  F(x_1,\ldots,x_N)= \vol{\mathop{\rm
      conv}\left\{B_{\rho_1}(x_1),\ldots,B_{\rho_N}(x_N)\right\}}
\end{equation}
satisfies the Steiner convexity property.

In \cite{paopivprob}, the functional in \eqref{eqn:Groemer_funct} was
generalzed to include operations beyond the convex hull. Namely, let
$C\subseteq \mathbb{R}^N$ be a compact convex set; for
$x_1,\ldots,x_N$, we view the $n\times N$ matrix $[x_1,\ldots,x_N]$ as
an operator from $\mathbb{R}^N$ to $\mathbb{R}^n$. Then
\begin{equation*}
  [x_1,\ldots,x_N]C =\Bigl\{\sum\nolimits_i
    c_ix_i:c=(c_i)\in C\Bigr\}
\end{equation*}and
\begin{equation}
  \label{eqn:XC}
  F(x_1,\ldots,x_N)= \vol{[x_1,\ldots,x_N]C}
\end{equation}
is Steiner convex \cite{paopivprob}. In particular, the functionals in
the stochastic Brunn-Minkowski inequality
\eqref{Probability-Brunn-Minkowski} and Groemer's inequality
\eqref{eqn:Groemer} fit naturally in this framework. Indeed, for
independent random vectors $X_1\ldots,X_N$ in $K$, and
$X_{N+1}$,$\ldots$, $X_{N+M}$ in $L$, we have
\begin{equation*}
  [K]_N=[X_1,\ldots,X_N]C_N
\end{equation*}and
\begin{eqnarray}
  [K]_N+[L]_M & = & [X_1,\ldots,X_N]C_N + [X_{N+1},\ldots,X_{N+M}]C_M
  \nonumber\\ & = & [X_1,\ldots,X_{N+M}](C_N+\widehat{C}_M),
  \label{eqn:hat_sum}
\end{eqnarray}
where $C_k=\mathop{\rm conv}\{e_1,\ldots,e_k\}$ for $k=N,M$ and
$\widehat{C}_M=\mathop{\rm conv}\{e_{N+1},\ldots,e_{N+M}\}$.

As noted in \cite{paopivrand}, convex operations on points in
\eqref{eqn:XC} can be combined with Euclidean balls (as in
\eqref{eqn:Pfiefer}) by using the notion of {\it
  $\mathcal{M}$-addition}. The latter operation was studied in depth
by Gardner, Hug and Weil in \cite{GarHugWei_M} as a unifying framework
for operations in $L_p$ and Orlicz Brunn-Minkowski theory; see, e.g.,
\cite{GHW_14, LYZ_Orlicz} for further references and background. For
$\mathcal{M}\subseteq\R^N$ and subsets $K_1,\ldots, K_N$ in $\R^n$,
their {\it $\mathcal{M}$-combination} is defined by
\begin{eqnarray*}
  \oplus_{\mathcal{M}}(K_1,\ldots,K_N) &=& \Bigl\{
  \displaystyle\sum\limits_{i=1}^N m_ix_i : x_i\in K_i,
  (m_1,\ldots,m_n)\in \mathcal{M} \Bigr\}.
\end{eqnarray*}
When $K_1,\ldots, K_N$ are convex and $\mathcal{M}$ is convex and
contained in the positive orthant, then
$\oplus_{\mathcal{M}}(K_1,\ldots,K_N)$ is convex \cite[Theorem
  6.1]{GarHugWei_M}.  With this notation, for $C = \mathcal{M}$,
\begin{equation*}
  \oplus_{C}(B_{\rho_1}(x_1),\ldots,B_{\rho_N}(x_N)) = \Bigl\{
  \displaystyle\sum\limits_{i=1}^N c_iu_i :u_i\in B_{\rho_i}(x_i),
  c=(c_i)\in C \Bigr\}.
\end{equation*}
To connect with the bodies of revolution $\mathcal{K}^s_f\subseteq
\mathbb{R}^n\times \mathbb{R}^s$ defined in \S \ref{exact-approx}, we
use $\mathcal{M}$-combinations of $s$-dimenisonal Euclidean balls
lying orthogonal to $\mathbb{R}^n$.  

\begin{Prop}
  \label{SC-of-rotation-with-fixed-heights}
  Let $\rho_1,\ldots,\rho_N\in[0,\infty)$ and $C$ a convex set contained in
  the positive orthant.  Then the function $F:(\R^n)^N\rightarrow
  [0,\infty)$ defined by \begin{equation}\label{functional}
      F(x_1,\ldots,x_N) = \vol{\oplus_C(B^s_{\rho_1}(x_1),\ldots,
      B^s_{\rho_N}(x_N))}
      \end{equation}
      is Steiner convex.
\end{Prop}

\begin{proof}
  We suppose without loss of generality that $\theta=e_1$, where
  $\{e_1,\ldots,e_{n}\}$ denotes the standard basis in $\R^n$.  Let
  $Y=\{y_1,\ldots,y_N\}\subseteq\theta^{\perp}$. We define the function
  $F_{Y}:\mathbb{R}^N\rightarrow [0,\infty)$ by
      \begin{eqnarray*}
	F_Y(t)&= &\vol{\oplus_C(B^s_{\rho_1}((t_1,y_1)),\ldots,
        B^s_{\rho_N}((t_N,y_N))}\\ & = &
        \vol{\left\{\sum\nolimits_{i=1}^{N}
        c_i(t_i,y_i,\widehat{z}_i):\abs{\widehat{z}_i}
        \leq \rho_i\right\}}.
      \end{eqnarray*}

To check that $F_Y$ is even we note that the sets involved in the
expressions for $F_Y(t)$ and $F_Y(-t)$ are reflections of each other
about $\theta^{\perp}$ and hence have equal volume.

To prove convexity, let $r,t\in \mathbb{R}^N$. For $a=(a_i)$ in
$\{r,t,\frac{r}{2}+\frac{t}{2}\}$, we write
\begin{equation*}
  \oplus_C(\{B_{\rho_i}^s(a_i,y_i)\}) :=
  \oplus_C(B^s_{\rho_1}((a_1,y_1)),\ldots, B^s_{\rho_N}((a_N,y_N))).
  \end{equation*}
Let $P:\mathbb{R}^{n+s}\rightarrow \theta^{\perp}\times \mathbb{R}^s$
denote the orthogonal projection.  Note that 
\begin{equation*}
  D := P(\oplus_C(\{B^s_{\rho_i}(a_i,y_i)\}))
\end{equation*}
is independent of $a$. Thus we define as in {\S \ref{prelim}},
$u_a,\ell_a:D\rightarrow \mathbb{R}$ by
\begin{equation*}
  u_a(v)=u(\oplus_C(\{B_{\rho_i}^s(a_i,y_i)\}),v)
\end{equation*}
and
\begin{equation*}
\ell_a(v)=\ell(\oplus_C(\{B_{\rho_i}^s(a_i,y_i)\}),v).
  \end{equation*}
  Next, we set $u=\frac{1}{2}(u_r+u_t)$ and
  $\ell=\frac{1}{2}(\ell_r+\ell_t)$ and define
  \begin{equation*}
    E :=\{(\lambda,v):v\in D,\ell(v)\leq\lambda\leq u(v)\}.
  \end{equation*}
   We claim that
   \begin{equation}
     \label{tildeset}
     \oplus_C\left(\left\{B^s_{\rho_i}
     \left(\tfrac{r_i+t_i}{2},y_i\right)\right\}\right)
     \subseteq E.
   \end{equation}
   To see this, let $w\in
   \oplus_C\left(\left\{B^s_{\rho_i}\left(\tfrac{r_i+t_i}{2},y_i\right)\right\}\right)$
   so that for some $c\in C$,
   \begin{equation*}
     \sum_{i=1}^N c_i\left(\tfrac{r_i+t_i}{2},y_i,\widehat{z}_i\right)
     = \Bigl(\sum_{i=1}^N
     c_i\bigl(\tfrac{r_i+t_i}{2}\bigr),y,\widehat{z}\Bigr),
   \end{equation*}where $y=\sum_{i}c_i y_i$ and
   $\widehat{z} = \sum_i c_i\widehat{z}_i$ Thus for $a\in\{r,t\}$, we
   have
   \begin{equation*}
     \sum_{i=1}^N c_i(a_i,y_i,\widehat{z}_i) =
     \Bigl(\sum_{i=1}^N c_ia_i,y,\widehat{z}\Bigr)\in
     \oplus_C(\{B^s_{\rho_i}(a_i,y_i)\}),
   \end{equation*}
   hence $\ell_a(0,y,\widehat{z})\leq \sum_{i=1}^N c_ia_i \leq
   u_a(0,y,\widehat{z})$. Thus,
   \begin{equation*}
   \ell(0,y,\widehat{z}) \leq
   \frac{1}{2}\sum_{i=1}^N c_ir_i
   +\frac{1}{2}\sum_{i=1}^N c_it_i \leq u(0,y,\widehat{z}),
   \end{equation*}
   which shows that $w\in E$ and establishes (\ref{tildeset}). Hence
\begin{eqnarray*}
  \vol{E} &=& \int_{D}{(u(v)-\ell(v))}dv\\ &=&
  \frac{1}{2}\int_{D}{(u_r-\ell_r)(v)}dv+
  \frac{1}{2}\int_{D}{(u_t-\ell_t)(v)}dv\\ &=&
  \frac{1}{2}\vol{\oplus_C(\{B_{\rho_i}^s(r_i,y_i)\})}+\frac{1}{2}
  \vol{\oplus_C(\{B_{\rho_i}^s(t_i,y_i)\})},
\end{eqnarray*}
which completes the proof.
\end{proof}

\section{Main proofs}

\begin{proof}
  [Proof of Theorem \ref{Theo-Stochastic-Groemer}] For
  $(x_1,z_1),\ldots,(x_N,z_N) \in \mathbb{R}^n\times [0,\infty)$,
    write $w_i=(x_i,z_i)$ and let $T_{\{w_i\}}$ be the least
    $\log$-concave function supported on the set $\mathbb{\rm
      conv}\{x_1,\ldots,x_N\}$ with $T(x_i)\geq z_i$, $i=1,\ldots,N$,
    i.e.,
    \begin{equation}
      \label{eqn:Tleastlog}
      T_{\{w_i\}}(x)=\exp(\sup\{z:(x,z)\in{\rm conv}\{(x_1,\log
      z_1),\ldots,(x_N,\log z_N)\}).
    \end{equation}
    With this notation, we set
    \begin{equation*}
      F(w_1,\ldots,w_N)=\int_{\R^n} T_{\{w_i\}}(x)dx.
    \end{equation*}
    Let $f:\R^n\to[0,\infty)$ be an integrable $\log$-concave
      function.  Sample independent random vectors $W_i=(X_i,Z_i)$,
      $i=1,\ldots, N$, according to the uniform Lebesgue measure on
      $G_f$ (cf. \eqref{eqn:graph}). Then the random function $[f]_N$
      defined as  in \eqref{Random-Least-Log-Concave} satisfies
    \begin{eqnarray*}
    {\Pro\left(\displaystyle\int_{\R^n}{[f]_N(x)}dx >\alpha \right)} &
    = & \mathbb{E} \mathds{1}_{\{ F>\alpha\}}(W_1,\ldots, W_N)\\
    & = & {\frac{1}{\norm{f}_1^N}\displaystyle
      \int_{N}{\mathds{1}_{\{F>\alpha\}}(\overline{w})}
      \prod\limits_{i=1}^N{\mathds{1}_{[0,f(x_i)]}(z_i)}}d\overline{w},
    \end{eqnarray*}
    where $\int_N$ denotes the integral on $(\R^n\times[0,\infty))^N$ and
    \begin{equation}
      \label{eqn:w}
      \overline{w}=(w_1,\ldots,w_N), \quad d\overline{w} = dw_1\ldots
      dw_N.
    \end{equation}

    It is sufficient to prove the theorem for
    $f_{\varepsilon}(x):=f(x)\mathds{1}_{\{f>\varepsilon\}}(x)$ for
    each fixed $\varepsilon >0$. To see this, note that
    $(f_{\varepsilon})^* = (f^*)_{\varepsilon}$. Thus if
    $W_i^{\varepsilon} = (X^{\varepsilon}_i,Z_i^{\varepsilon})$,
    $i=1,\ldots,N$ are independent random vectors distributed
    according to the uniform Lebesgue measure on
    $G_{f_{\varepsilon}}$, we have
    \begin{eqnarray*}
    {\Pro\left(\displaystyle\int_{\R^n}{[f_{\varepsilon}]_N(x)}dx
      >\alpha \right)} &= & \mathbb{E}
    \mathds{1}_{\{F>\alpha\}}(W_1^{\varepsilon},\ldots,W^{\varepsilon}_{N})\\ &
    = & {\frac{1}{\norm{f_{\epsilon}}_1^N}\displaystyle
      \int_{N}{\mathds{1}_{\{F\}>\alpha\}}(\overline{w})}\prod\limits_{i=1}^N{\mathds{1}_{[\varepsilon,f(x_i)]}(z_i)}}d\overline{w}.
  \end{eqnarray*}
Since $\norm{f_{\varepsilon}}_1\rightarrow \norm{f}_1$ and
$f_{\varepsilon}\leq f$, dominated convergence implies
\begin{equation*}
  \Pro\left( \displaystyle\int_{\R^n}{[f]_N(x)}dx >\alpha \right)=
  \underset{\epsilon\to0}{\lim}~\Pro\left(
  \displaystyle\int_{\R^n}{[f_{\varepsilon}]_N(x)}dx >\alpha \right).
\end{equation*}
For $s>0$ and $x\in \mathbb{R}^n$, we define
\begin{equation}
  \label{Stochastic-s-approx}
  [f_{\epsilon}]_{N,s}(x) = \left(
  1+\frac{\log{[f_{\epsilon}]_{N}(x)}}{s} \right)_{+}^s.
\end{equation}
For every $x$, we have almost sure convergence
$[f_{\varepsilon}]_{N,s}(x) \rightarrow [f_{\varepsilon}]_{N}(x)$ as
$s\rightarrow \infty$. Since $[f_{\varepsilon}]_{N,s},
[f_{\varepsilon}]_{N}$ are dominated by $f$, this implies almost sure
convergence of the integrals
\begin{equation}
  \label{eqn:as_int}
  \int_{\mathbb{R}^n} [f_{\varepsilon}]_N(x) dx =
  \lim\limits_{s \rightarrow 0}\int_{\mathbb{R}^n}
             [f_{\varepsilon}]_{N,s}(x)dx.
\end{equation}
The latter implies that 
    \begin{equation*}
      \Pro\left(\int_{\mathbb{R}^n} [f_{\varepsilon}]_N(x)dx \leq
      \alpha\right) = \lim_{s\rightarrow
        \infty}\Pro\left(\int_{\mathbb{R}^n}
                  [f_{\varepsilon}]_{N,s}(x)dx \leq \alpha \right) 
    \end{equation*}at all continuity points of 
    \begin{equation}
      \label{eqn:feps}
      [0,\infty)\ni a\mapsto \Pro\left(\int_{\mathbb{R}^n}
        [f_{\varepsilon}]_N(x)dx \leq a\right).
  \end{equation}
However, since $N>n+1$ and $f_{\varepsilon}$ is absolutely continuous
with respect to Lebesgue measure, the random variable
$\int_{\mathbb{R}^n}[f_{\varepsilon}]_{N}(x)dx$ is positive almost
surely and \eqref{eqn:feps} is continuous on $[0,\infty)$.

Thus it is sufficient to prove the theorem for
$[f_{\varepsilon}]_{N,s}$ with $s$ large enough. For
$s>-\log{\epsilon}$, we have
$R_i:=R_i(Z_1^{\varepsilon}):=(1+\frac{\log{{Z}^{\varepsilon}_i}}{s})>0$ for
all $i\in\{1,\ldots,N\}$. Then
\begin{equation*}
\left[f_{\epsilon}\right]_{N,s}(x) = \sup\{z^s\in\R:(x,z)\in
H_{f_{\epsilon},N,s}\},
  \end{equation*}where 
\begin{eqnarray*}
  H_{f_{\epsilon},N,s} &:=& {\rm
    conv}\left\{\left(X_1^{\varepsilon},R_1(Z_1^{\varepsilon})\right),
  \ldots,\left(X_N^{\varepsilon},
  R_N(Z_N^{\varepsilon})\right)\right\}.
\end{eqnarray*} 
By Lemma \ref{lemma:rotation},
\begin{equation*}
  \mathcal{K}^s_{[f_{\epsilon}]_{N,s}} = \mathop{\rm conv}\{
  B^s_{R_1}({X}_1^{\varepsilon}), \ldots,
  B^s_{R_N}({X}_N^{\varepsilon})\}.
\end{equation*}
By \eqref{Volume-Integral},
\begin{equation}
  \label{St-Volume-Integral}  
  \int_{\mathbb{R}^n} [f_{\varepsilon}]_{N,s}(x)dx =
  \kappa_s^{-1}\vol{\mathcal{K}^s_{[f_{\varepsilon}]_{N,s}}}.
\end{equation}
For $z_1,\ldots,z_N\in [0,\infty)$, we write
$\rho_i=\rho_i(z_i)=(1+\tfrac{\log z_i}{s})$ and define
\begin{equation*}
  F_s((x_1,z_1),\ldots,(x_N,z_N)) = \vol{\mathop{\rm conv}\{
  B^s_{\rho_1}(x_1), \ldots, B^s_{\rho_N}(x_N)\}}.
\end{equation*} 
Then for each fixed $z_1,\ldots,z_N$,
\begin{equation*}
  (x_1,\ldots,x_N)\mapsto F_s((x_1,z_1),\ldots,(x_N,z_N))
\end{equation*}
 is a Steiner convex function by Proposition
 \ref{SC-of-rotation-with-fixed-heights}. Applying Fubini, writing
 $d\overline{w} = d\overline{x}d\overline{z}$ and invoking Theorem
 \ref{RBLLC}, we obtain
\begin{eqnarray*}
  \Pro\left( \vol{\mathcal{K}^s_{[f_{\varepsilon}]_{N,s}}} >\alpha \right)
  &=& \frac{1}{\norm{f_{\varepsilon}}_1^N}\displaystyle
  \int_{N}\mathds{1}_{\{F_s>\alpha\}}(\overline{w})
    \prod\limits_{i=1}^N\mathds{1}_{[\varepsilon,f(x_i)]}(z_i)
    d\overline{w}\\ &=&
    \frac{1}{\norm{f_{\varepsilon}}_1^N}\int_{[0,\infty)^N}{\left(
        \int_{(\R^n)^N}{\mathds{1}_{\{F_s>\alpha\}}
          \prod\limits_{i=1}^N{\mathds{1}_{[\varepsilon,f(x_i)]}(z_i)}}d\overline{x}
        \right)}d\overline{z}\\ &\geq &
      \frac{1}{\norm{f_{\varepsilon}^*}_1^N}\int_{[0,\infty)^N}{\left(
          \int_{(\R^n)^N}{\mathds{1}_{\{F_s>\alpha\}}
            \prod\limits_{i=1}^N{\mathds{1}_{[\varepsilon,f^*(x_i)]}(z_i)}}d\overline{x}
          \right)}d\overline{z}\\ &= &
        \frac{1}{\norm{f_{\varepsilon}^*}_1^N}\displaystyle
        \int_{N}\mathds{1}_{\{F_s>\alpha\}}(\overline{w})
          \prod\limits_{i=1}^N\mathds{1}_{[\varepsilon,f^*(x_i)]}(z_i)
          d\overline{w}\\ & = &\Pro\left(
          \vol{\mathcal{K}^s_{[f^*_{\varepsilon}]_{N,s}}} >\alpha \right).
\end{eqnarray*}	
The result now follows from \eqref{St-Volume-Integral} applied to
$f_{\varepsilon}^*$.
\end{proof}

\begin{proof}[Proof of Theorem~\ref{Theo-Stochastic-PL}]
  For $w_i=(x_i,z_i)\in \mathbb{R}^n\times\mathbb{R}^s$,
  $i=1,\ldots,M+N$, we set 
  \begin{equation*}
    F(w_1,\ldots,w_{N+M})=\int_{\R^n}T^{(N)}_{\{w_i\}}\star_{\lambda}
    T^{(M)}_{\{w_i\}}(v)dv,
  \end{equation*}
  where $T^{(N)}_{\{w_i\}}$ and $T^{(M)}_{\{w_i\}}$ are the least
  $\log$-concave functions above the collections $\{w_i\}_{i\leq N}$
  and $\{w_i\}_{N+1\leq i\leq M}$, respectively (as defined as in
  \eqref{eqn:Tleastlog}).  Then the stochastic approximations
  $[f]_{N}$, $[g]_{M}$ to $f$, $g$ satisfy
  \begin{eqnarray*}
    \lefteqn{\Pro\left(\int_{\R^n}[f]_N\star_{\lambda}[g]_M(v)
      dv>\alpha\right)}\\
    & & = \frac{1}{\prod_{i=1}^{M+N}\norm{h_i}_1}
    \int_{N+M}\mathds{1}_{\{F>\alpha\}}(\overline{w})
      \prod\limits_{i=1}^{N+M}\mathds{1}_{[0,h_i(x_i)]}(z_i) d\overline{w},
  \end{eqnarray*}
  where $\int_{N+M}$ is the integral on
  $(\R^n\times[0,\infty))^{N+M}$, $\overline{w}$ and $d\overline{w}$
    are as in \eqref{eqn:w}, $h_i=f_i$ for $i=1,\ldots, N$ and
    $h_i=g_i$ for $i=N+1,\ldots,N+M$. For $\varepsilon >0$, we apply
    the latter identity with $f_{\varepsilon}$ and $g_{\varepsilon}$
    and use dominated convergence to get
    \begin{equation*}
      \Pro\left(\int_{\R^n}[f]_N\star_{\lambda}[g]_M(v) dv>\alpha\right)
      = \lim\limits_{\varepsilon \rightarrow 0}
      \Pro\left(\int_{\R^n}[f_{\epsilon}]_N\star_{\lambda}[g_{\epsilon}]_M(v)
      dv>\alpha\right).
    \end{equation*}
For $\varepsilon >0$, we sample independent random vectors
$\{(X_i^{\varepsilon},Z_i^{\varepsilon})\}_{i=1}^{N}$ uniformly in
$G_{f_{\varepsilon}}$ and
$\{(X_i^{\varepsilon},Z_i^{\varepsilon})\}_{i=N+1}^{N+M}$ uniformly in
$G_{g_{\varepsilon}}$. For $s>-\log \varepsilon$, we define
$[f_{\epsilon}]_{N,s}$, and $[g_{\epsilon}]_{M,s}$ as in
\eqref{Stochastic-s-approx}. Note that for each $v\in \mathbb{R}^n$,
we have almost sure convergence \begin{equation*}
  ([f_{\epsilon}]_N\star_{\lambda}[g_{\epsilon}]_M)(v) =
  \lim\limits_{s\to\infty}([f_{\epsilon}]_{N,s}\star_{\lambda,s}[g_{\epsilon}]_{M,s})(v).
\end{equation*}
Since $[f_{\epsilon}]_{N,s}\star_{\lambda,s}[g_{\epsilon}]_{M,s}$ and
$[f_{\epsilon}]_{N}\star_{\lambda}[g_{\epsilon}]_{M}$ are dominated
by $f\star_{\lambda} g$, we can argue as in the proof of Theorem
\ref{Theo-Stochastic-Groemer} to get, for each $\alpha\geq 0$,
\begin{eqnarray*}
  \lefteqn{\Pro\left(
  \int_{\mathbb{R}^n}([f_{\epsilon}]_N\star_{\lambda}[g_{\epsilon}]_M)(v)dv\leq
  \alpha \right)}\\
  & & = \lim\limits_{s \rightarrow \infty} \Pro\left(
  \int_{\mathbb{R}^n}([f_{\epsilon}]_{N,s}\star_{\lambda,s}[g_{\epsilon}]_{M,s})(v)dv\leq
  \alpha \right).
\end{eqnarray*}
Since $[f_{\epsilon}]_{N,s}\star_{\lambda,s}[g_{\epsilon}]_{M,s}$ is
$s$-concave, the body
$\mathcal{K}^s_{[f_{\epsilon}]_{N,s}\star_{\lambda,s}[g_{\epsilon}]_{M,s}}$
is convex. Moreover, by \eqref{Volume-Integral},
\begin{equation}
  \label{eqn:star_lam}
  \int_{\mathbb{R}^n}([f_{\epsilon}]_{N,s}\star_{\lambda,s}[g_{\epsilon}]_{M,s})(v)dv =
  \kappa_s^{-1}\vol{K_{[f_{\epsilon}]_{N,s}\star_{\lambda,s}[g_{\epsilon}]_{M,s}}}.
\end{equation}
Recalling \eqref{eqn:Ksum}, we have
\begin{eqnarray*}
  \mathcal{K}^s_{[f_{\epsilon}]_{N,s}\star_{\lambda,s}[g_{\epsilon}]_{M,s}}
  & = &\lambda \mathcal{K}^s_{[f_{\epsilon}]_{N,s}}+ (1-\lambda)
  \mathcal{K}^s_{[g_{\varepsilon}]_{M,s}} \\& = & 
  \mathcal{K}^s_{[f_{\epsilon}]_{N,s}}+_{\lambda}\mathcal{K}^s_{[g_{\varepsilon}]_{M,s}}.
\end{eqnarray*}
Set $R_i:=R_i(Z_i):=\left(1 +\frac{\log Z_i^{\varepsilon}}{s}\right)$.  By Lemma
\ref{lemma:rotation}, we have
\begin{eqnarray*}
  \mathcal{K}^s_{[f_{\epsilon}]_{N,s}}
& =& \mathop{\rm
    conv}\{B^s_{R_1}(X_1),\ldots,B^s_{R_N}(X_N)\}\\
  &= & \oplus_{C_N}(\{B^s_{R_i}(X_i)\}_{i=1}^N),
\end{eqnarray*} where $C_N={\rm conv}\{e_1,\ldots,e_N\}$.  Similarly,
\begin{eqnarray*}
  \mathcal{K}^s_{[g_{\epsilon}]_{M,s}} &=& \mathop{\rm
    conv}\{B^s_{R_{N+1}}(X_{N+1}),\ldots,B^s_{R_{N+M}}(X_{N+M})\}\\ &=
  & \oplus_{C_M}(\{B^s_{R_i}(X_i)\}_{i=N+1}^M),
\end{eqnarray*}where $C_M={\rm conv}\{e_1,\ldots,e_M\}$.
Thus if we write $\widehat{C}_M ={\rm
  conv}\{e_{N+1},\ldots,e_{N+M}\}$, (which is similar to
\eqref{eqn:hat_sum}), we have
\begin{eqnarray*}
  \mathcal{K}^s_{[f_{\epsilon}]_{N,s}\star_{\lambda,s}[g_{\epsilon}]_{M,s}}
  &=& \oplus_{C_N}(\{B^s_{R_i}(X_i)\}_{i=1}^N) +_{\lambda}
  \oplus_{C_M}(\{B^s_{R_i}(X_i)\}_{i=N+1}^{N+M})\\ &=&
  \oplus_{C_N+_{\lambda}
    \widehat{C}_M}(\{B^s_{R_i}(X_i)\}_{i=1}^{N+M}).
    \end{eqnarray*}
For $z_1,\ldots,z_{N+M}\in [0,\infty)$, write
$\rho_i=\rho_i(z_i)=(1+\tfrac{\log z_i}{s})$ and define
\begin{equation*}
  F_s((x_1,z_1),\ldots,(x_{N+M},z_{N+M})) =
  \vol{\oplus_{C_N+_{\lambda}\widehat{C}_M}(\{B^s_{\rho_i}(x_i)\}_{i=1}^{N+M})}.
\end{equation*} Then for each fixed $z_1,\ldots,z_{N+M}$,  the function
\begin{eqnarray*}
(x_1,\ldots,x_{N+M})\mapsto F_s((x_1,z_1),\ldots,(x_{N+M},z_{N+M}))
\end{eqnarray*} 
is Steiner convex by Proposition
\ref{SC-of-rotation-with-fixed-heights}.  By Fubini and Theorem
\ref{RBLLC}, we have
\begin{eqnarray*}
  \lefteqn{\Pro\left(
    \vol{\mathcal{K}^s_{[f_{\varepsilon}]_{N,s}\star_{\lambda,s}
        [g_{\varepsilon}]_{N,s}}} >\alpha \right)}\\ & &
  =\frac{1}{\prod_{i=1}^{N+M}\norm{(h_i)_{\varepsilon}}_1}\displaystyle
  \int_{N+M}\mathds{1}_{\{F_s>\alpha\}}(\overline{w})
  \prod\limits_{i=1}^{N+M}\mathds{1}_{[\varepsilon,h_i(x_i)]}(z_i)
  d\overline{w}\\ & & =
  \frac{1}{\prod_{i=1}^{N+M}\norm{(h_i)_{\varepsilon}}_1}\int_{[0,\infty)^{N+M}}{\left(
      \int_{(\R^n)^{N+M}}{\mathds{1}_{\{F_s>\alpha\}}(\overline{w})
        \prod\limits_{i=1}^{N+M}{\mathds{1}_{[\varepsilon,h_i(x_i)]}(z_i)}}d\overline{x}
      \right)}d\overline{z}\\ & & \geq
    \frac{1}{\prod_{i=1}^{N+M}\norm{(h_i)_{\varepsilon}^*}_1}
    \int_{[0,\infty)^{N+M}}{\left(
        \int_{(\R^n)^{N+M}}{\mathds{1}_{\{F_s>\alpha\}}(\overline{w})
          \prod\limits_{i=1}^{N+M}{\mathds{1}_{[\varepsilon,h_i^*(x_i)]}(z_i)}}d\overline{x}
        \right)}d\overline{z}\\ & & =
      \frac{1}{\prod_{i=1}^{N+M}\norm{(h_i)_{\varepsilon}^*}_1}\displaystyle
      \int_{N+M}\mathds{1}_{\{F_s>\alpha\}}(\overline{w})
      \prod\limits_{i=1}^{N+M}\mathds{1}_{[\varepsilon,h_i^*(x_i)]}(z_i)
      d\overline{w}\\ & & = \Pro\left(
      \vol{\mathcal{K}^s_{[f^*_{\varepsilon}]_{N,s}\star_{\lambda,s}[g^*_{\varepsilon}]_{N,s}}}
      >\alpha \right).
\end{eqnarray*}	
The result now follows from \eqref{eqn:star_lam} applied to
$f_{\varepsilon}^*$ and $g_{\varepsilon}^*$.

\end{proof}

\noindent {{\bf Acknowledgments} It is our pleasure to thank Dario
  Cordero-Erausquin, Matthieu Fradelizi, Mokshay Madiman and Grigoris
  Paouris for helpful correspondence, discussions and feedback.
  P.\ Pivovarov was supported by NSF grant DMS-1612936 and Simons
  Foundation grant \#635531.}

%
%

\addcontentsline{toc}{section}{References} \bibliographystyle{plain}
\bibliography{biblio1}

\begin{thebibliography}{10}

\bibitem{floartseg}
S.~Artstein-Avidan, D.I. Florentin, and A.~Segal.
\newblock {P}olar {P}r\'ekopa-{L}eindler inequalities.
\newblock Preprint. \url{https://arxiv.org/abs/1707.08732}.

\bibitem{artklamilsan}
S.~Artstein-Avidan, B.~Klartag, and V.~Milman.
\newblock The {S}antal\'{o} point of a function, and a functional form of the
  {S}antal\'{o} inequality.
\newblock {\em Mathematika}, 51(1-2):33--48 (2005), 2004.

\bibitem{AAKSW}
S.~Artstein-Avidan, B.~Klartag, C.~Sch\"{u}tt, and E.~Werner.
\newblock Functional affine-isoperimetry and an inverse logarithmic {S}obolev
  inequality.
\newblock {\em J. Funct. Anal.}, 262(9):4181--4204, 2012.

\bibitem{ball-phd}
K.~Ball.
\newblock {\em Isoperimetric problems in $\ell_p$ and sections of convex sets}.
\newblock PhD thesis, Cambridge, 1986.

\bibitem{BarFra_13}
F.~Barthe and M.~Fradelizi.
\newblock The volume product of convex bodies with many hyperplane symmetries.
\newblock {\em Amer. J. Math.}, 135(2):311--347, 2013.

\bibitem{Bla:1917}
W.~Blaschke.
\newblock \"{U}ber affine geometrie xi: L\"{o}sung des “vierpunktproblems”
  von {S}ylvester aus der theorie der geometrischen wahrscheinlichkeiten.
\newblock {\em Leipziger Berichte}, 69:436--453, 1917.

\bibitem{bobcolfraque}
S.~G. Bobkov, A.~Colesanti, and I.~Fragal\`a.
\newblock Quermassintegrals of quasi-concave functions and generalized
  {P}r\'{e}kopa-{L}eindler inequalities.
\newblock {\em Manuscripta Math.}, 143(1-2):131--169, 2014.

\bibitem{borconv}
C.~Borell.
\newblock Convex set functions in {$d$}-space.
\newblock {\em Period. Math. Hungar.}, 6(2):111--136, 1975.

\bibitem{braliebestconstants}
H.~J. Brascamp and E.~H. Lieb.
\newblock Best constants in {Y}oung's inequality, its converse, and its
  generalization to more than three functions.
\newblock {\em Advances in Math.}, 20(2):151--173, 1976.

\bibitem{bralieonextension}
H.~J. Brascamp and E.~H. Lieb.
\newblock On extensions of the {B}runn-{M}inkowski and {P}r\'{e}kopa-{L}eindler
  theorems, including inequalities for log concave functions, and with an
  application to the diffusion equation.
\newblock {\em J. Funct. Anal.}, 22(4):366--389, 1976.

\bibitem{BLL}
H.~J. Brascamp, E.~H. Lieb, and J.~M. Luttinger.
\newblock A general rearrangement inequality for multiple integrals.
\newblock {\em J. Functional Analysis}, 17:227--237, 1974.

\bibitem{burashort}
A.~Burchard.
\newblock A short course on rearrangement inequalities.
\newblock Avaible at \url{http://www.math.utoronto.ca/almut/rearrange/pdf}.

\bibitem{Bus:1953}
H.~Busemann.
\newblock Volume in terms of concurrent cross-sections.
\newblock {\em Pacific J. Math.}, 3:1--12, 1953.

\bibitem{CFGLSW}
U.~Caglar, M.~Fradelizi, O.~Gu\'{e}don, J.~Lehec, C.~Sch\"{u}tt, and E.~M.
  Werner.
\newblock Functional versions of {$L_p$}-affine surface area and entropy
  inequalities.
\newblock {\em Int. Math. Res. Not. IMRN}, 2016(4):1223--1250, 2016.

\bibitem{camcolgroanote}
S.~Campi, A.~Colesanti, and P.~Gronchi.
\newblock A note on {S}ylvester's problem for random polytopes in a convex
  body.
\newblock {\em Rend. Istit. Mat. Univ. Trieste}, 31(1-2):79--94, 1999.

\bibitem{christestimates}
M.~Christ.
\newblock Estimates for the {$k$}-plane transform.
\newblock {\em Indiana Univ. Math. J.}, 33(6):891--910, 1984.

\bibitem{CE_PL}
D.~Cordero-Erausquin.
\newblock On matrix-valued log-concavity and related {P}r\'{e}kopa and
  {B}rascamp-{L}ieb inequalities.
\newblock {\em Adv. Math.}, 351:96--116, 2019.

\bibitem{CEK_12}
D.~Cordero-Erausquin and B.~Klartag.
\newblock Interpolations, convexity and geometric inequalities.
\newblock In {\em Geometric aspects of functional analysis}, volume 2050 of
  {\em Lecture Notes in Math.}, pages 151--168. Springer, Heidelberg, 2012.

\bibitem{CEM_PL}
D.~Cordero-Erausquin and B.~Maurey.
\newblock Some extensions of the {P}r\'{e}kopa-{L}eindler inequality using
  {B}orell's stochastic approach.
\newblock {\em Studia Math.}, 238(3):201--233, 2017.

\bibitem{CEMS}
D.~Cordero-Erausquin, R.~J. McCann, and M.~Schmuckenschl\"{a}ger.
\newblock A {R}iemannian interpolation inequality \`a la {B}orell, {B}rascamp
  and {L}ieb.
\newblock {\em Invent. Math.}, 146(2):219--257, 2001.

\bibitem{frameyfun}
M.~Fradelizi and M.~Meyer.
\newblock Some functional forms of {B}laschke-{S}antal\'{o} inequality.
\newblock {\em Math. Z.}, 256(2):379--395, 2007.

\bibitem{FraMey_08}
M.~Fradelizi and M.~Meyer.
\newblock Some functional inverse {S}antal\'{o} inequalities.
\newblock {\em Adv. Math.}, 218(5):1430--1452, 2008.

\bibitem{garbru}
R.~J. Gardner.
\newblock The {B}runn-{M}inkowski inequality.
\newblock {\em Bull. Amer. Math. Soc. (N.S.)}, 39(3):355--405, 2002.

\bibitem{GHW_14}
R.~J. Gardner, D.~Hug, and W.~Weil.
\newblock The {O}rlicz-{B}runn-{M}inkowski theory: a general framework,
  additions, and inequalities.
\newblock {\em J. Differential Geom.}, 97(3):427--476, 2014.

\bibitem{GarHugWei_M}
R.J. Gardner, D.~Hug, and W.~Weil.
\newblock Operations between sets in geometry.
\newblock {\em J. Eur. Math. Soc. (JEMS)}, 15(6):2297--2352, 2013.

\bibitem{GT_radius}
A.~Giannopoulos and A.~Tsolomitis.
\newblock Volume radius of a random polytope in a convex body.
\newblock {\em Math. Proc. Cambridge Philos. Soc.}, 134(1):13--21, 2003.

\bibitem{gromean}
H.~Groemer.
\newblock On the mean value of the volume of a random polytope in a convex set.
\newblock {\em Arch. Math. (Basel)}, 25:86--90, 1974.

\bibitem{klamarg}
B.~Klartag.
\newblock Marginals of geometric inequalities.
\newblock In {\em Geometric aspects of functional analysis}, volume 1910 of
  {\em Lecture Notes in Math.}, pages 133--166. Springer, Berlin, 2007.

\bibitem{klamillog}
B.~Klartag and V.~D. Milman.
\newblock Geometry of log-concave functions and measures.
\newblock {\em Geom. Dedicata}, 112:169--182, 2005.

\bibitem{lehdir}
J.~Lehec.
\newblock A direct proof of the functional {S}antal\'{o} inequality.
\newblock {\em C. R. Math. Acad. Sci. Paris}, 347(1-2):55--58, 2009.

\bibitem{leiconv}
L.~Leindler.
\newblock On a certain converse of {H}\"{o}lder's inequality. {II}.
\newblock {\em Acta Sci. Math. (Szeged)}, 33(3-4):217--223, 1972.

\bibitem{lielosanalysis}
E.~H. Lieb and M.~Loss.
\newblock {\em Analysis}, volume~14 of {\em Graduate Studies in Mathematics}.
\newblock American Mathematical Society, Providence, RI, second edition, 2001.

\bibitem{LYZ_Orlicz}
E.~Lutwak, D.~Yang, and G.~Zhang.
\newblock Orlicz centroid bodies.
\newblock {\em J. Differential Geom.}, 84(2):365--387, 2010.

\bibitem{madmelzu}
M.~Madiman, J.~Melbourne, and P.~Xu.
\newblock Forward and reverse entropy power inequalities in convex geometry.
\newblock In {\em Convexity and concentration}, volume 161 of {\em IMA Vol.
  Math. Appl.}, pages 391--425. Springer, New York, 2017.

\bibitem{mel}
J.~Melbourne.
\newblock Rearrangement and pr\'ekopa-leindler type inequalities.
\newblock Avaible at \url{https://arxiv.org/abs/1806.08837}.

\bibitem{milrotmix}
V.~Milman and L.~Rotem.
\newblock Mixed integrals and related inequalities.
\newblock {\em J. Funct. Anal.}, 264(2):570--604, 2013.

\bibitem{paopivprob}
G.~Paouris and P.~Pivovarov.
\newblock A probabilistic take on isoperimetric-type inequalities.
\newblock {\em Adv. Math.}, 230(3):1402--1422, 2012.

\bibitem{paopivrand}
G.~Paouris and P.~Pivovarov.
\newblock Randomized isoperimetric inequalities.
\newblock In {\em Convexity and concentration}, volume 161 of {\em IMA Vol.
  Math. Appl.}, pages 391--425. Springer, New York, 2017.

\bibitem{Pfiefer}
R.~E. Pfiefer.
\newblock {\em The extrema of geometric mean values}.
\newblock ProQuest LLC, Ann Arbor, MI, 1982.
\newblock Thesis (Ph.D.)--University of California, Davis.

\bibitem{preonlog1}
A.~Pr\'{e}kopa.
\newblock Logarithmic concave measures with application to stochastic
  programming.
\newblock {\em Acta Sci. Math. (Szeged)}, 32:301--316, 1971.

\bibitem{preonlog2}
A.~Pr\'{e}kopa.
\newblock On logarithmic concave measures and functions.
\newblock {\em Acta Sci. Math. (Szeged)}, 34:335--343, 1973.

\bibitem{rinonconvexity}
Y.~Rinott.
\newblock On convexity of measures.
\newblock {\em Ann. Probability}, 4(6):1020--1026, 1976.

\bibitem{Rogers_single}
C.~A. Rogers.
\newblock A single integral inequality.
\newblock {\em J. London Math. Soc.}, 32:102--108, 1957.

\bibitem{rogsheext}
C.~A. Rogers and G.~C. Shephard.
\newblock Some extremal problems for convex bodies.
\newblock {\em Mathematika}, 5:93--102, 1958.

\bibitem{schcon}
R.~Schneider.
\newblock {\em Convex bodies: the {B}runn-{M}inkowski theory}, volume 151 of
  {\em Encyclopedia of Mathematics and its Applications}.
\newblock Cambridge University Press, Cambridge, expanded edition, 2014.

\bibitem{simconvexity}
B.~Simon.
\newblock {\em Convexity}, volume 187 of {\em Cambridge Tracts in Mathematics}.
\newblock Cambridge University Press, Cambridge, 2011.
\newblock An analytic viewpoint.

\end{thebibliography}

\end{document}